\documentclass[12pt,reqno]{amsart}
\usepackage{geometry}
\usepackage{xspace}
\usepackage{tikz-cd,tikz}
\usepackage{amsmath, amssymb, amsfonts, mathtools}

\usepackage{xcolor}
\definecolor{darkgreen}{rgb}{0,0.45,0}
\definecolor{darkred}{rgb}{0.6,0,0}
\definecolor{darkblue}{rgb}{0,0,0.7}

\usepackage[
    colorlinks, 
    citecolor=blue, 
    urlcolor=darkgreen, 
    final, 
    hyperindex, 
    pagebackref,
    linkcolor = darkblue
]{hyperref}

\usepackage[capitalize]{cleveref}

\usepackage{amsthm}

\theoremstyle{plain}
\newtheorem{thm}{Theorem}
\newtheorem*{thm*}{Theorem}

\newtheorem{lem}[thm]{Lemma}
\newtheorem*{prop*}{Proposition}

\newtheorem{conj}[thm]{Conjecture}
\newtheorem*{theorem1}{Theorem 1}

\theoremstyle{definition}
\newtheorem{defn}[thm]{Definition}
\newtheorem{ex}[thm]{Example}

\theoremstyle{remark}

\newtheoremstyle{named}{}{}{\itshape}{}{\bfseries}{.}{.5em}{\thmnote{#3's }#1}
\theoremstyle{named}

\newcommand{\define}[1]{{\bf \boldmath{#1}}\index{#1}}

\newcommand{\john}[1]{\textcolor{purple}{#1}}
\newcommand{\todd}[1]{\textcolor{darkgreen}{#1}}

\DeclareFontFamily{U}{min}{}
\DeclareFontShape{U}{min}{m}{n}{<-> udmj30}{}

\newcommand{\maps}{\colon}
\newcommand{\To}{\Rightarrow}

\newcommand{\GL}{\mathrm{GL}}
\newcommand{\SL}{\mathrm{SL}}
\newcommand{\Sp}{\mathrm{Sp}}
\newcommand{\SO}{\mathrm{SO}}
\newcommand{\OO}{\mathrm{O}}
\newcommand{\M}{\mathrm{M}}
\renewcommand{\O}{\mathcal{O}}

\newcommand{\End}{\mathrm{End}}

\newcommand{\namedcat}[1]{\mathsf{#1}}
\mathchardef\mhyphen="2D

\newcommand{\Alg}{\namedcat{Alg}}\newcommand{\Aff}{\namedcat{AffSch}}
\newcommand{\Bialg}{\namedcat{Bialg}}

\newcommand{\Comm}{\namedcat{Comm}}

\newcommand{\Fin}{\namedcat{Fin}}

\newcommand{\LinCat}{\namedcat{LinCat}}

\newcommand{\Coalg}{\namedcat{Coalg}}
\newcommand{\Cocomm}{\namedcat{Cocomm}}
\newcommand{\ksbar}{\overline{k\S}}

\newcommand{\Rep}{\namedcat{Rep}}

\newcommand{\tRig}{2\mhyphen\namedcat{Rig}}
\newcommand{\Vect}{\namedcat{Vect}}

\newcommand{\C}{\namedcat{C}}
\newcommand{\D}{\namedcat{D}}

\newcommand{\I}{\namedcat{I}}

\renewcommand{\P}{\namedcat{P}}
\newcommand{\R}{\namedcat{R}}
\renewcommand{\S}{\namedcat{S}}

\newcommand{\AAff}{\namedbicat{AffSch}}

\newcommand{\TRig}{2\mhyphen\namedbicat{Rig}}

\newcommand{\im}{\mathrm{im}}

\newcommand{\Comod}{\mathsf{Comod}}

\newcommand{\op}{^\mathrm{op}}

\renewcommand{\hom}{\mathrm{hom}}
\newcommand{\Sym}{\mathrm{Sym}}

\newcommand{\namedbicat}[1]{\mathbf{#1}}
\newcommand{\FFin}{\namedbicat{Fin}}
\newcommand{\Cat}{\namedbicat{Cat}}

\newcommand{\Lin}{\namedbicat{Lin}}

\begin{document}

\title[Tannaka Reconstruction and the Monoid of Matrices]{Tannaka Reconstruction \\ and the Monoid of Matrices}

\author[Baez]{John C.\ Baez$^{1}$} 
\author[Trimble]{\\Todd Trimble$^2$}

\address{$^1$Department of Mathematics, University of California, Riverside, CA 92521, USA}

\address{$^2$Department of Mathematics, Western Connecticut State University, Danbury, CT 06810, USA}

\email{baez@math.ucr.edu, trimblet@wcsu.edu}

\begin{abstract}
Settling a conjecture from an earlier paper, we prove that the monoid $\M(n,k)$ of $n \times n$ matrices in a field $k$ of characteristic zero is the `walking monoid with an $n$-dimensional representation'.   More precisely, if we treat $\M(n,k)$ as a monoid in affine schemes, the 2-rig $\Rep(\M(n,k))$ of algebraic representations of $\M(n,k)$ is the free 2-rig on an object $x$ with $\Lambda^{n+1}(x) \cong 0$.  Here a `2-rig' is a symmetric monoidal $k$-linear category that is Cauchy complete.  Our proof uses Tannaka reconstruction and a general theory of quotient 2-rigs and 2-ideals.  We conclude with a series of conjectures about the universal properties of representation 2-rigs of classical groups.
\end{abstract}

\maketitle
\setcounter{tocdepth}{1} 
\tableofcontents

\section{Introduction}
\label{sec:intro} 

The monoid $\M(n,k)$ of $n \times n$ matrices with entries in a field $k$, with matrix multiplication as its monoid structure, plays a distinguished role in linear algebra and representation theory.  It is the `walking monoid with an $n$-dimensional representation'.  At a superficial level, this says that any $n$-dimensional representation of any monoid factors through the tautologous representation of $\M(n,k)$ on $k^n$.  But there is another deeper sense in which this is true.

The Tannakian philosophy directs us to study an algebraic structure through its category of representations.  Here we show, roughly speaking, that the category of algebraic representations of $\M(n,k)$ is the free 2-rig on an object of dimension $n$.  But this statement needs some clarification.

A 2-rig is a categorified form of a rig, or `ring without negatives'.  In a 2-rig of representations, addition is the direct sum of representations while multiplication is the tensor product.  There are various ways to make this idea precise, but for our present purposes we define a 2-rig over a field $k$ to be a symmetric monoidal $k$-linear category that is Cauchy complete.  Examples include categories of vector bundles, group representations, coherent sheaves, and so on.  In our previous papers \cite{Schur,Splitting} we developed the theory of such 2-rigs, showing that if the field $k$ has characteristic zero, the free 2-rig on one object is a semisimple  category whose simple objects correspond to Young diagrams.  Here we further develop that theory and apply it to the 2-rig of representations of $\M(n,k)$.

However, two words of clarification are in order. First, if we treat $\M(n,k)$ as mere monoid, its category of representations on finite-dimensional vector space over $k$ can be extremely complicated.  Already when $n = 1$, each automorphism of the field $k$ gives a different one-dimensional representation, so we are led into issues of Galois theory.  To eliminate these complexities, we limit ourselves to `algebraic' representations.  Very concretely, an algebraic representation of $\M(n,k)$ on the vector space $k^m$ is a monoid homomorphism $\rho \maps \M(n,k) \to \M(m,k)$ such that each matrix entry of $\rho(x)$ is a \emph{polynomial} in the entries of $\M(n,k)$.   There is a 2-rig of algebraic representations of $\M(n,k)$, which we denote by $\Rep(\M(n,k))$.   Our goal is to show that $\Rep(\M(n,k))$ is characterized by an appealing universal property.

Second, our work \cite{Splitting} has uncovered several choices of what it can mean for an object of a 2-rig to have dimension $n$.  One can define the exterior powers $\Lambda^n(x)$ of any object $x$ in any 2-rig over a field of characteristic zero.  We then say an object $x$ has `bosonic dimension $n$' if $\Lambda^n(x)$ has an inverse with respect to the tensor product, and `bosonic subdimension $n$' if $\Lambda^{n+1}(x) \cong 0$.  For example, a vector bundle on a topological space may have different ranks on different connected components.  It has bosonic subdimension $n$ precisely when it has rank at most $n$ on every component, and bosonic dimension $n$ when it has rank exactly $n$ on every component.  The theory of 2-rigs has a kind of built-in `supersymmetry', so there are also fermionic versions of dimension and subdimension defined using symmetric powers, but we do not need these here, so we omit the adjective `bosonic' from now on.

Our main theorem is a strengthened version of an earlier conjecture \cite[Conj.\ 8.7]{Splitting}:

\begin{thm}
\label{thm:free_2-rig_on_object_of_subdimension_N} 
If $k$ is a field of characteristic zero, the 2-rig $\Rep(\M(n,k))$ is the free 2-rig on an object of subdimension $n$.  That is, given any 2-rig $\R$ containing an object $r$ of subdimension $n$, there is a map of 2-rigs $F \maps \Rep(\M(n,k)) \to \R$ with $F(k^n) = r$, unique up to a uniquely determined monoidal natural isomorphism.
\end{thm}
\noindent
Here a \define{map of 2-rigs} is a symmetric monoidal $k$-linear functor between 2-rigs.  

The proof goes roughly as follows.  We begin with the the free 2-rig on one object.  We call this $\ksbar$, since it can be obtained by the following three-step process \cite{Schur}:
\begin{itemize}
    \item First form the free symmetric monoidal category on one object $x$. This is equivalent to the groupoid of finite sets and bijections, which we call $\S$, with disjoint union providing the symmetric monoidal structure.
    \item Then form the free $k$-linear symmetric monoidal category on $\S$ by freely forming $k$-linear combinations of morphisms. This is called $k\S$.
    \item Then Cauchy complete $k\S$. The result, $\ksbar$, is the coproduct, as Cauchy complete $k$-linear categories, of the categories of finite-dimensional representations of all the symmetric groups $S_n$.
\end{itemize}

We then construct the free 2-rig on an object of subdimension $n$ by taking the `quotient' of $\ksbar$ by the `2-ideal' generated by the object $\Lambda^{n+1}(x)$, or $\Lambda^{n+1}$ for short.  We call this quotient 2-rig $\ksbar/\langle \Lambda^{n+1} \rangle$.  Of course, to define and work with this 2-rig we need to develop the theory of 2-ideals and quotient 2-rigs.   In the process we prove that $\ksbar/\langle \Lambda^{n+1} \rangle$ is a semisimple abelian category.

By the universal property of $\ksbar/\langle \Lambda^{n+1} \rangle$, there is a 2-rig map
\[
    \begin{tikzcd}
        \ksbar/ \langle \Lambda^{n+1} \rangle \ar[d, "j"] \\
        \Fin\Vect,
    \end{tikzcd}
\]
unique up to isomorphism, sending $x$ to the vector space $k^n$.  We prove that $j$ is faithful and exact.

Then we use Tannaka reconstruction.  To begin with, this theory implies that if a 2-rig $\R$ is an abelian category equipped with a faithful exact 2-rig map 
\[
    \begin{tikzcd}
        \R \ar[d, "U"] \\
        \Fin\Vect
    \end{tikzcd}
\]
then $\R$ is equivalent to the 2-rig $\Comod(B)$ of finite-dimensional comodules of some commutative bialgebra $B(U)$.  Furthermore, this bialgebra can be constructed as a coend
\[    B(U) = \int^{r \in \R} U(r) \otimes U(r)^\ast .\]
By calculating this coend in the case at hand we show that
\[    B(j) \cong S(k^n \otimes (k^n)^\ast)  \]
where $S$ stands for the symmetric algebra.   To understand the bialgebra structure on $B(j)$, note that the vector space $k^n \otimes (k^n)^\ast$ is isomorphic to $\M(n,k)^\ast$, so that its symmetric algebra is isomorphic to the algebra of functions on $\M(n,k)$ that are polynomials in the matrix entries:
\[     S(k^n \otimes (k^n)^\ast) \cong \O(\M(n,k)). \]
The latter algebra becomes a bialgebra with comultiplication coming from matrix multiplication, and we show that this gives the relevant bialgebra structure on $B(j)$.  We thus conclude that
\[    \ksbar/\langle \Lambda^{n+1} \rangle \simeq \Comod(\O(\M(n,k)))  \]
as 2-rigs.  Finally, we use the fact that algebraic representations of $\M(n,k)$ are equivalent to finite-dimensional comodules of $\mathcal{O}(\M(n,k))$, and obtain the desired equivalence of 2-rigs:
\[        \ksbar/\langle \Lambda^{n+1} \rangle \simeq \Rep(\M(n,k)) . \]

Here is the plan of the paper.  In \cref{sec:representation} we review the theory of affine monoids and their algebraic representations.   \cref{lem:commuting_triangle_of_2-rigs} makes precise the sense in which every every affine monoid $M$ has a 2-rig of representations that is equivalent to the 2-rig of finite-dimensional comodules of a cocommutative bialgebra $\O(M)$, not merely as abstract 2-rigs, but as 2-rigs over $\Fin\Vect$.  We pay special attention to the case where $M = \M(n,k)$ is the `full linear monoid' of the vector space $k^n$: the monoid of $n \times n$ matrices, treated as an affine monoid.  

In \cref{sec:Tannaka} we review the theory of Tannaka reconstruction.  This leads up to Theorems \ref{thm:Tannaka_1'} and \ref{thm:Tannaka_2'}, concerning an adjunction between the category of cocommutative bialgebras and the category of 2-rigs over $\Fin\Vect$.  These theorems give conditions under which a 2-rig over $\Fin\Vect$ is equivalent to the category of finite-dimensional comodules of some cocommutative bialgebra.  

In \cref{sec:quotient_2-rigs} we develop the theory of quotient 2-rigs and 2-ideals.   In \cref{thm:quotient_2-rig} we show that when a 2-rig $\R$ is also a semisimple abelian category, any 2-rig map $F \maps \R \to \R'$ factors as $\R \to \P \to \R'$ where $\P$ is the quotient of $\R$ by a 2-ideal, the `kernel' of $F$, and $\P$ itself is a semisimple abelian category. The 2-rig map $\R \to \P$ is essentially surjective and full, while $\P \to \R'$ is faithful.

In \cref{sec:quotients_of_the_free_2-rig} we apply the results of the previous section to 2-rig maps $\ksbar \to \Fin\Vect$.  We show that up to isomorphism there is one such map for each natural number, with the $n$th one, say $\phi_n \maps \ksbar \to \Fin\Vect$, sending the generating object $x \in \ksbar$ to the vector space $k^n$.  We show that the kernel of $\phi_n$ is the 2-ideal $\langle \Lambda^{n+1} \rangle$ and that the 2-rig $\ksbar/\langle \Lambda^{n+1} \rangle$ is a semisimple abelian category.

In \cref{sec:main_theorem} we prove our main theorem, \cref{thm:free_2-rig_on_object_of_subdimension_N}.  To do this, in \cref{lem:free_2-rig} we prove the intuitively obvious fact that $\ksbar/\langle \Lambda^{n+1} \rangle$ is the free 2-rig on an object of subdimension $n$.  Then we carry out the Tannaka reconstruction argument sketched above, to show that this 2-rig is equivalent to $\Rep(\M(n,k))$.

Finally, in \cref{sec:conclusions} we state some conjectures concerning the universal properties of the 2-rigs of representations of various classical groups.

\vskip 1em
\subsection*{Notation}
We use sans-serif font for 1-categories such as $\Vect$, and bold serif font for 2-categories such as $\TRig$.

\section{Representations of affine monoids}
\label{sec:representation}

There is a long tradition of studying representations of linear algebraic groups \cite{Borel,Humphreys,MilneAlgGp}, but more recently this line of work has been generalized to monoids  \cite{Brion, Renner}.   A `linear algebraic monoid' is simply a set of $n \times n$ matrices over $k$, closed under matrix multiplication, that is picked out by a collection of polynomial equations in the matrix entries.  

Formal aspects of the theory become easier if we work more generally with `affine monoids', namely  monoid objects in the category of affine schemes over $k$.  

\begin{defn}
The category $\Aff$ of \define{affine schemes over $k$}, or \define{affine schemes} for short, is the opposite of the category $\Comm\Alg$ of commutative algebras over $k$. 
\end{defn}

\begin{defn}
An \define{affine monoid scheme over $k$}, or  \define{affine monoid} for short, is a monoid internal to $(\Aff, \times)$. 
\end{defn}

Any finite-dimensional algebra over $k$ gives an affine monoid.  To see this, note that there is a symmetric lax monoidal functor
\[   \Phi \maps (\Fin\Vect, \otimes) \to (\Aff, \times) \]
given as the composite
\[    (\Fin\Vect, \otimes) \xrightarrow{(-)^\ast} 
(\Fin\Vect, \otimes)\op \xrightarrow{\Sym\op} (\Comm\Alg, \otimes)\op = (\Aff, \times). \]
In the first step, taking the dual is a symmetric strong monoidal functor from $(\Fin\Vect, \otimes)$ to $ (\Fin\Vect, \otimes)\op$. In the second step, $\Sym \maps (\Fin\Vect, \otimes) \to (\Comm\Alg, \otimes)$ sends any vector space $V$ to the free commutative algebra on $V$, also known as the symmetric algebra $\Sym(V)$.   This functor $\Sym$ is symmetric oplax monoidal since it is left adjoint to the forgetful functor $U \maps (\Comm\Alg, \otimes) \to (\Fin\Vect, \otimes)$, which is strong symmetric monoidal.   Thus $\Sym\op \maps (\Fin\Vect, \otimes)\op \to (\Comm\Alg, \otimes)\op$ is symmetric lax monoidal.  It follows that the composite $\Phi$ is symmetric lax monoidal.  Thanks to this fact, $\Phi$ sends internal monoids to internal monoids, so any finite-dimensional algebra $A$ over $k$ gives an affine monoid $\Phi(A)$.   

This construction yields the primordial example of an affine monoid, which is the subject of this paper:

\begin{defn}
For any finite-dimensional vector space $V$ over $k$, the \define{full linear monoid} $\M(V)$ is the affine monoid obtained by applying $\Phi$ to the algebra of endomorphisms of $V$.  In particular, the full linear monoid $\M(n,k)$ is the affine monoid obtained by applying $\Phi$ to the algebra of $n \times n$ matrices over $k$.
\end{defn}

To define the category of algebraic representations of an affine monoid, we can use the symmetric lax monoidal functor $\Phi$ to convert categories enriched in $\Fin\Vect$ into categories enriched in $\Aff$.

\begin{defn} 
\label{defn:finite-dim_linear_category}
    Let $\FFin\Lin\Cat$ be the 2-category of categories, functors and natural transformations enriched over $\Fin\Vect$. We call these \define{finite-dimensional linear} categories, \define{linear} functors and natural transformations. 
\end{defn}

\begin{defn}
\label{defn:algebraic_category}
    Let $\AAff\Cat$ be the 2-category of categories, functors and natural transformations enriched over $\Aff$. We call these \define{affine categories}, \define{algebraic functors} and natural transformations.
\end{defn}

\begin{lem}
\label{lem:base_change}
    Base change along $\Phi$ gives a 2-functor
    \[   (-)^\sim \, \maps
    \FFin\Lin\Cat \to \AAff\Cat \]
    sending any finite-dimensional linear category $\C$ to the affine category $\C^\sim$ that has the same objects, with hom-objects defined by 
    \[   \C^\sim(x,y) = \Phi(\C(x,y)),\]
    and composition and units defined using the functoriality of $\Phi$.
\end{lem}

\begin{proof}  This is Lemma 5.8 of \cite{Splitting}.
\end{proof}

\begin{defn} 
\label{defn:algebraic_representation}
Given an affine category $\C$ let 
\[   \Rep(\C) = \AAff\Cat(\C, \Fin\Vect^\sim). \]
The objects of $\Rep(\C)$ are algebraic functors $F \maps \C \to \Fin\Vect^\sim$, which we call \define{algebraic representations} of $\C$, and the morphisms are natural transformations between these, which we call \define{maps} between algebraic representations.
\end{defn}

We are especially interested in algebraic representations of affine monoids, which can be seen as one-object affine categories.  For this it is important that the category of affine monoids is the opposite of the category of commutative bialgebras over $k$.  To see this, note that $(\Aff, \times) \simeq (\Comm\Alg, \otimes)\op$, where $\otimes$ denotes the tensor product of commutative algebras over $k$, which is their coproduct. But a monoid in $(\Comm\Alg, \otimes)\op$ is the same as a comonoid in $(\Comm\Alg, \otimes)$, and the latter is exactly a commutative bialgebra.   Thus, any affine monoid has a corresponding commutative bialgebra, which we call its \define{coordinate bialgebra} $\mathcal{O}(M)$.

\begin{lem}
\label{lem:commuting_triangle_of_categories}
The category $\Rep(M)$ of algebraic representations of an affine monoid $M$ is equivalent, as a $k$-linear category over $\Fin\Vect$, to the category $\Comod(\O(M))$ of finite-dimensional comodules of its coordinate bialgebra.  There is a strictly commuting triangle of $k$-linear functors
\[
\begin{tikzcd}[column sep=-0.7em]
    \Rep(M) 
    \ar{rr}{\cong}[swap]{E}
    \ar[dr,swap,"W"]
    & &
    \Comod(\O(M))
    \ar[dl,"U"]
     \\
   & \Fin\Vect
\end{tikzcd}
\]
where $E$ an isomorphism and $U$ (resp.\ $W$) is the forgetful functor sending a comodule (resp.\ algebraic representation) to its underlying vector space.
\end{lem}

\begin{proof} 
An equivalence between $\Rep(M)$ and $\Comod(\O(M))$ was constructed in the proof of Lemma 5.14 of \cite{Splitting}, but in fact the equivalence constructed there is an isomorphism making the above triangle strictly commute, since it sends any algebraic representation of $M$ on a finite-dimensional vector space $V$ to a comodule with the same underlying vector space.
\end{proof}

Next we enhance the triangle in \cref{lem:commuting_triangle_of_categories} to a commuting triangle of 2-rig maps.  First, note that for any commutative bialgebra $C$, $\Comod(C)$ with the usual tensor product of comodules is an \define{abelian 2-rig}: a 2-rig that is also an abelian category.  Second, note that the forgetful functor
\[        U \maps \Comod(C) \to \Fin\Vect  \]
is faithful and exact (preserves exact sequences).  It becomes a symmetric strict monoidal $k$-linear functor if we define the symmetric monoidal structure on $\Comod(C)$ as follows: the tensor product of two comodules $(V, \gamma \maps V \to V \otimes C)$ and $(V', \gamma' \maps V' \to V' \otimes C)$ is $V \otimes V'$ at the underlying vector space level, with comodule structure being the obvious composite 
\[
\begin{tikzcd}
    V \otimes V' \ar[r, "\gamma \otimes \gamma'"] & V \otimes C \otimes V' \otimes C \ar[r, "1 \otimes \sigma \otimes 1"] & V \otimes V' \otimes C \otimes C \ar[r, "1 \otimes 1 \otimes m"] & V \otimes V' \otimes C
\end{tikzcd}
\]
where $\sigma$ denotes a symmetry isomorphism and $m$ denotes the algebra multiplication for $C$. The symmetry $V \otimes V' \cong V' \otimes V$ at the level of underlying vector spaces is a comodule map if commutativity of $m$ is assumed.  It therefore follows that 
\[        U \maps \Comod(C) \to \Fin\Vect  \]
is a (strict monoidal) map of 2-rigs.   Summarizing:

\begin{lem}
\label{lem:Comod}
For any commmutative bialgebra $C$, $U \maps \Comod(C) \to \Fin\Vect$ is strict monoidal, exact, faithful 2-rig map between abelian 2-rigs.
\end{lem}

Next suppose $C = \O(M)$ for an affine monoid $M$.  
If we transfer the 2-rig structure from $\Comod(\O(M))$ to $\Rep(M)$ using the isomorphism $E$ in \cref{lem:commuting_triangle_of_categories} we get the following result:

\begin{lem}
\label{lem:commuting_triangle_of_2-rigs}
For any affine monoid $M$,
\[
\begin{tikzcd}[column sep=0em]
    \Rep(M) 
    \ar{rr}{\cong}[swap]{E}
    \ar[dr,swap,"W"]
    & &
    \Comod(\O(M))
    \ar[dl,"U"]
     \\
   & \Fin\Vect
\end{tikzcd}
\]
is a strictly commuting triangle where all the arrows are strict monoidal exact maps between abelian 2-rigs.  $U$ and $V$ are faithful, while $E$ is an isomorphism.
\end{lem}

The central focus of this paper is the case where $M$ is the full linear monoid $\M(N,k)$.  In \cite[Ex. 5.10]{Splitting} we discussed the coordinate bialgebra of this affine monoid, and we expand on that material here.

\begin{ex}
\label{ex:full_linear_monoid}
The vector space of $N \times N$ matrices has a basis of elementary matrices, so its dual has a dual basis, say $e^i_j$ for $1 \le i, j \le N$.  The coordinate bialgebra $\O(\M(n,k))$ is the polynomial algebra on these elements $e^i_j$, with comultiplication given by
\[   \Delta(e^i_j) = \sum_{k = 1}^N e^i_k \otimes e^k_j \] 
and counit given by
\[   \varepsilon(e^i_j) = \delta^i_j \]
where $\delta$ is the Kronecker delta.

This example also has a useful basis-independent description.  To give this we begin by associating to any finite-dimensional vector space $V$ a coalgebra $V^\ast \otimes V$.   

\begin{defn}
\label{defn:coendomorphism_coalgebra}
For any finite-dimensional vector space $V$, the dual of the algebra of endomorphisms $\Fin\Vect(V, V) \cong V \otimes V^\ast$ is the
\define{coendomorphism coalgebra}
\[    (V \otimes V^\ast)^\ast \cong V^\ast \otimes V .\] 
\end{defn}

To explicitly describe the comultiplication and counit in the coendomorphism coalgebra, we use the fact that any finite-dimensional vector space $V$ and its dual $V^\ast$ are equipped with a counit
\begin{equation}
\label{eq:varepsilon}
\begin{array}{cccl}
\varepsilon \maps & V^\ast \otimes V & \to & k \\
            & f \otimes v & \mapsto & f(v) 
\end{array}
\end{equation}
and unit
\begin{equation}
\label{eq:eta}
\begin{array}{ccccl}
\eta \maps & k & \to & \Fin\Vect(V,V) & \cong V \otimes V^\ast \\
& 1 & \mapsto & 1_V  
\end{array}
\end{equation}
obeying the triangle equations in the definition of an adjunction.  An adjunction in any bicategory gives a comonad in that bicategory.  As a consequence, the adjunction between $V$ and $V^\ast$ makes into a coalgebra whose comultiplication $\Delta$ is the composite
\[
\begin{tikzcd}
V^\ast \otimes V \arrow[r, "\cong"]
\arrow[rr, bend left = 25, "\Delta"] 
& 
 V^\ast \otimes k \otimes V 
\arrow[r,"1 \otimes \eta \otimes 1"] 
& 
V^\ast \otimes V \otimes V^\ast \otimes V
\end{tikzcd}
\]
and whose counit is $\varepsilon$.

The coordinate bialgebra of $\M(V)$ is canonically isomorphic to the free commutative algebra on $V^\ast \otimes V$:
\[  \O(\M(V)) \cong \Sym((V \otimes V^\ast)^\ast)
\cong \Sym(V^\ast \otimes V)  .\]
Since $\Sym \maps \Vect \to \Vect$ is oplax monoidal, it maps coalgebras to coalgebras.  This makes $\Sym(V^\ast \otimes V)$ and thus $\O(\M(V))$ into a coalgebra.   

Thus, the comultiplication and counit for $V^\ast \otimes V$ extend uniquely to algebra homomorphisms
\[    \Sym(V^\ast \otimes V) \to \Sym(V^\ast \otimes V) \otimes \Sym(V^\ast \otimes V) , \]
\[   \Sym(V^\ast \otimes V) \to k \]
which we again call these $\Delta$ and $\varepsilon$, and these give $O(\M(V)) \cong \Sym(V^\ast \otimes V)$ its coalgebra structure.  
\end{ex}

\section{Tannaka reconstruction} 
\label{sec:Tannaka}

Thanks to \cref{lem:commuting_triangle_of_2-rigs}, to prove our main theorem we just need to show that the free 2-rig on an object of subdimension $N$ is the 2-rig of finite-dimensional comodules of the commutative bialgebra $\O(\M(N,k))$.   For this we use Tannaka reconstruction.  This characterizes 2-rigs of finite-dimensional comodules of commutative bialgebras, and lets us reconstruct a commutative bialgebra from its 2-rig of finite-dimensional comodules. 

Here we recall the facts we need about Tannaka reconstruction; for more details see the papers by Deligne \cite{Deligne}, Deligne--Milne \cite{DeligneMilne}, and Joyal--Street \cite{JoyalStreet}.   We start with a simplified version that works for coalgebras, and then turn to commutative bialgebras.

Let $k$ be a field.  Let $\LinCat \downarrow \Fin\Vect$ be the category whose objects are $k$-linear categories $\C$ equipped with a $k$-linear functor $U \maps \C \to \Fin\Vect$, and whose morphisms $(\C, U) \to (\D, V)$ are $k$-linear functors $F \maps \C \to \D$ such that $U = V \circ F$ (strictly). Let $\Coalg$ be the category of coalgebras over $k$. There is a functor 
\[
\Comod \maps \Coalg \to \LinCat \downarrow \Fin\Vect
\]
taking a coalgebra $C$ to the category $\Comod(C)$ of its finite-dimensional right comodules together with the forgetful functor $U \maps \Comod(\C) \to \Fin\Vect$. As we shall see, for more or less tautological reasons, there is a left adjoint to $\Comod$, 
\[  
\End^\vee \maps \LinCat \downarrow \Fin\Vect \to \Coalg
\]
taking an object $(\C, U)$ to a coalgebra that will be denoted as $\End(\C, U)^\vee$, or just $\End(U)^\vee$. Tannaka reconstruction for coalgebras then comes in two parts. Here is the first part:

\begin{thm}[\textbf{Tannaka Reconstruction for Coalgebras 1}]
\label{thm:Tannaka_1} 
The counit of the adjunction $\End^\vee \dashv \Comod$, evaluated at any coalgebra $C$, is an isomorphism \hfill \break  $\End(\Comod(C), U)^\vee \to C$. 
\end{thm} 

For the second part, observe that $\Comod(C)$ is an abelian category for any coalgebra $C$, and that the forgetful functor $U \maps \Comod(U) \to \Fin\Vect$ is faithful and exact. 

\begin{thm}[\textbf{Tannaka Reconstruction for Coalgebras 2}] 
\label{thm:Tannaka_2}
The unit of the adjunction $\End^\vee \dashv \Comod$, evaluated at any pair $(\C, U)$ where $\C$ is an abelian category and $U$ is faithful and exact, is an equivalence of $k$-linear categories $\C \to \Comod(\End(U)^\vee)$.
\end{thm} 

Proofs of these results can be found in the references \cite{Deligne,DeligneMilne,JoyalStreet}.  Here we just describe the adjunction between $\Comod$ and $\End$, which we need to concretely reconstruct a coalgebra from its 2-rig of representations.

For any coalgebra $C$, let $\Comod(C)$ denote its category of finite-dimensional right comodules, and $U \maps \Comod(C) \to \Fin\Vect$ the forgetful functor to finite-dimensional spaces.  Each object $c \in \Comod(C)$ has a comodule structure, which is a linear map 
\[
\gamma_c \maps U(c) \to U(c) \otimes C ,
\]
which induces a linear map 
\[  \phi_c \maps U(c)^\ast \otimes U(c) \to C \]
where $U(c)^\ast \otimes U(c)$ is the coendomorphism coalgebra of $U(c)$, as introduced in \cref{defn:coendomorphism_coalgebra}.  Here $\phi_c$
is the composite 
\[
\begin{tikzcd}
    U(c)^\ast \otimes U(c) 
    \ar[r, "1 \otimes \gamma_c"] & 
    U(c)^\ast \otimes U(c) \otimes C 
    \ar[r, "\varepsilon \otimes 1"] & 
    k \otimes C \cong C
\end{tikzcd}
\]
where $\varepsilon$ is defined as in \cref{eq:varepsilon}. In fact, linear maps of type $\phi_c$ correspond bijectively to linear maps of type $\gamma_c$, using the adjunction isomorphism
\[\hom(U(c)^\ast \otimes Uc, C)  \;\, \cong \;\; \hom(U(c), U(c) \otimes C).\]
The condition that $\gamma \maps U(c) \to C \otimes U(c)$ is a comodule structure is equivalent to the condition that $\phi \maps U(c)^\ast \otimes U(c) \to C$ is a coalgebra map; this is dual to how $R$-module structures $R \otimes V \to V$ correspond to $R$-algebra maps $R \to \hom(V, V)$. Thus, comodule structures $\gamma_c \maps U(c) \to U(c) \otimes C$ are in bijection with coalgebra maps $\phi_c \maps U(c)^\ast \otimes U(c) \to C$. 

Moreover, the assignment $c \mapsto \gamma_c$ is a transformation that is natural with respect to comodule maps $f \maps c \to d$, simply by virtue of the definition of comodule map. It follows that the indexed family of coalgebra maps $c \mapsto \phi_c$, indexed over objects $c$, is \emph{dinatural} with respect to comodule maps $f \maps c \to d$.  This family therefore corresponds to a uniquely determined linear map 
\[
\Phi_C \maps \int^{c: \Comod(C)} U(c)^\ast \otimes U(c) \to C.
\]
and this linear map is a coalgebra map, with respect to the coalgebra structure on the coend whose comultiplication 
\[
\left[\int^c U(c)^\ast \otimes U(c)\right] \longrightarrow \left[\int^c U(c)^\ast \otimes U(c)\right] \otimes \left[\int^c U(c)^\ast \otimes U(c)\right]
\]
corresponds to the dinatural family expressed by the composite 
\[
U(c)^\ast \otimes U(c) \overset{1 \otimes \eta \otimes 1}{\longrightarrow} U(c)^\ast \otimes U(c) \otimes U(c)^\ast \otimes U(c) \overset{i \otimes i}{\longrightarrow} \left[\int^c U(c)^\ast \otimes U(c)\right] \otimes \left[\int^c U(c)^\ast \otimes U(c)\right]
\]
(using $i$ to denote a coend coprojection). That $\Phi_C$ is indeed a coalgebra map is immediate from the definition of the coalgebra structure on the coend, together with the fact that the components $\phi_c$ are coalgebra maps. 

Now we reenact this argument but do it more generally. Suppose we are given a small $k$-linear category $\C$ together with a $k$-linear functor $F \maps \C \to \Fin\Vect$. To give a morphism $(\C, F) \to (\Comod(C), U)$ in $\LinCat \downarrow \Fin\Vect$, i.e., a functor $G \maps \C \to \Comod(C)$ such that $F = U \circ G$, or in other words a lift $G$ of $F$ through the forgetful functor from comodules to vector spaces, is precisely to endow $F$ with a $C$-comodule structure. That is to say, it is precisely to give a transformation $F \to F \otimes C$, i.e., a family of linear maps
\[
F(c) \to F(c) \otimes C
\]
natural in objects $c$ of $\C$, and obeying the axioms for an internal $C$-comodule structure in $[\C, \Fin\Vect]$. By the same reasoning as in the preceding paragraph, such comodule structures are in natural bijection with coalgebra maps 
\[
\int^{c: \C} F(c)^\ast \otimes F(c) \to C
\]
and this coend defines the construction $\End(\C, F)^\vee$. Again, this coend may be calculated at the level of vector spaces; its coalgebra structure is determined, by dinaturality, from the coalgebra structures on the $F(c)^\ast \otimes F(c)$. 

This argument shows that for each coalgebra $C$ and each object $(\C, F \maps \C \to \Fin\Vect)$ in $\LinCat \downarrow \Fin\Vect$, there is a natural bijection between 
\begin{itemize}
    \item maps $(\C, F) \to (\Comod(C), U)$ 
\end{itemize}
and 
\begin{itemize}
    \item coalgebra maps $\End(\C, F)^\vee \to C$,
\end{itemize} 
and this establishes the Tannaka adjunction. The canonical map 
\[
\Phi_C \maps \int^{c: \Comod(C)} U(c)^\ast \otimes U(c) \to C
\]
established above is the counit (at the coalgebra $C$) of this adjunction.

Tannaka reconstruction for commutative bialgebras is a refinement of the ideas we have just seen.   Let $\tRig \downarrow \Fin\Vect$ be the category whose objects are 2-rigs $\C$ equipped with a 2-rig map $U \maps \C \to \Fin\Vect$, and whose morphisms $(\C, U) \to (\D, V)$ are 2-rig maps $F \maps \C \to \D$ such that $U = V \circ F$.  Let $\Comm\Bialg$ be the category of commutative bialgebras over $k$.  By \cref{lem:Comod}, any commutative bialgebra $\C$ gives an object of $\tRig \downarrow \Fin\Vect$, namely the forgetful functor
\[   U \maps \Comod(C) \to \Fin\Vect  \]
which is a strict monoidal, exact, and faithful map between abelian 2-rigs.   This construction is functorial, so it defines a functor
\[
\Comod \maps \Comm\Bialg \to \tRig \downarrow \Fin\Vect .
\]

This functor $\Comod$ has a left adjoint $\End^\vee$, as in the coalgebra case. The algebra structure on $\End(\C, U)^\vee$ is built on the condition that $U$ is a strong symmetric monoidal functor: its requisite algebra multiplication, 
\[
\left(\int^c U(c)^\ast \otimes U(c)\right) \otimes \left(\int^d U(d)^\ast \otimes U(d)\right) \to \int^e U(e)^\ast \otimes U(e),
\]
is the universal map induced by the evident dinatural family 
\[
\begin{tikzcd}[column sep=1.2em]
     {[Uc, Uc]} \otimes [Ud, Ud] \ar[r] & {[Uc \otimes Ud, Uc \otimes Ud]} \ar[rr, "{[\theta_{cd}^{-1}, \theta_{cd}]}"] & & {[U(c \otimes d), U(c \otimes d)]} \ar[r] & \int^e [Ue, Ue]
\end{tikzcd}
\]
where $\theta_{cd} \maps Uc \otimes Ud \to U(c \otimes d)$ is the structural constraint for $U$ to be symmetric monoidal, and $[V, W]$ denotes internal hom of vector spaces, and taking advantage of $Uc^\ast \otimes Uc \cong [Uc, Uc]$. 

The Tannaka reconstruction at the coalgebra level lifts to Tannaka reconstruction at the bialgebra level. That is to say, one side of the lifted adjunction concerns bialgebra maps of the form $
\End(\C, F)^\vee \to C$, while the other side concerns 2-rig maps $\C \to \Comod(C) $ that preserve the given 2-rig maps down from these 2-rigs to $\Fin\Vect$. The assertion is that there is a natural bijection between these classes of maps. 

To see this is just a matter of checking that the diagrams of coalgebra maps expressing preservation of multiplication and unit (respectively), correspond to diagrams expressing how $G \maps \C \to \Comod(C)$, the linear lift  of the given 2-rig map $F \maps \C \to \Fin\Vect$, respects the monoidal product and monoidal unit of $C$-comodules (respectively). And this in turn is just a matter of unpacking the definitions. On one side, the condition that 
\[
\int^c F(c)^\ast \otimes F(c) \to C
\]
preserves multiplication amounts to the assertion that a family of diagrams of coalgebras of the form 
\[
\begin{tikzcd}
    {[F(c), F(c)]} \otimes {[F(d), F(d)]} \ar[r, "\phi_c \otimes \phi_d"] \ar[d, "\cong"] & C \otimes C \ar[dd, "m"] \\
    {[Fc \otimes Fd, Fc \otimes Fd]} \ar[d, "\cong"] \\
    {[F(c \otimes d), F(c \otimes d)]} \ar[r, swap, "\phi_{c \otimes d}"] & C
\end{tikzcd}
\]
are commutative (where the bottom left vertical arrow uses $\theta_{cd} \maps Fc \otimes Fd \to F(c \otimes d)$ and its inverse). If $\tilde{\phi_c} \maps F(c) \to F(c) \otimes C$ denotes the comodule structure mated to the coalgebra map $\phi_c \maps F(c)^\ast \otimes F(c) \to C$, then the diagram above is mated to a diagram of the form 
\[
\begin{tikzcd}
    Fc \otimes Fd \ar[r, "\tilde{\phi_c} \otimes \tilde{\phi_d}"] & Fc \otimes C \otimes Fd \otimes C \ar[r, "\cong"] & Fc \otimes Fd \otimes C \otimes C \ar[r, "1 \otimes 1 \otimes m"] & Fc \otimes Fd \otimes C \ar[d, "\cong"] \\
    F(c \otimes d) \ar[u, "\cong"] \ar[rrr, "\widetilde{\phi_{c \otimes d}}"] & & & F(c \otimes d) \otimes C
\end{tikzcd}
\]
where the vertical isomorphisms are $\theta_{cd}^{-1}$, $\theta_{cd}$, respectively. The bottom horizontal map expresses the structure of the comodule $G(c \otimes d)$, and the top composite expresses the structure of the comodule $G(c) \otimes G(d)$, hence $G$ preserves the monoidal product up to coherent isomorphism. The demonstration that $G$ preserves the monoidal unit is left to the reader. 

In this way, the Tannaka adjunction lifts to the level of bialgebras and 2-rigs.  We have a functor 
\[
\Comod \maps \Cocomm\Bialg \to \tRig \downarrow \Fin\Vect
\]
taking a cocommutative bialgebra $C$ to the category $\Comod(C)$ of its finite-dimensional comodules together with the forgetful functor $U \maps \Comod(\C) \to \Fin\Vect$.   This functor has a left adjoint
\[  
\End^\vee \maps \tRig \downarrow \Fin\Vect \to \Cocomm\Bialg,
\]
and the resulting adjunction has these properties:

\begin{thm}[\textbf{Tannaka Reconstruction for Commutative Bialgebras 1}]
\label{thm:Tannaka_1'} 
The counit of the adjunction $\End^\vee \dashv \Comod$, evaluated at any commutative bialgebra $C$, is an isomorphism $\End(\Comod(C), U)^\vee \to C$. 
\end{thm}

\begin{thm}[\textbf{Tannaka Reconstruction for Commutative Bialgebras 2}] 
\label{thm:Tannaka_2'}
The unit of the adjunction $\End^\vee \dashv \Comod$, evaluated at any pair $(\C, U)$ where $\C$ is abelian and $U$ is faithful and exact, is an equivalence of 2-rigs $\C \to \Comod(\End(U)^\vee)$.
\end{thm}

\section{Quotient 2-rigs} 
\label{sec:quotient_2-rigs}

We plan to construct the free 2-rig on an object of subdimension $n$ as a quotient of the free 2-rig on one object, $\ksbar$.  Since $\ksbar$ is semisimple, we use some results on quotients of semisimple 2-rigs, which can be summarized in this theorem:

\begin{thm}
\label{thm:quotient_2-rig}
If $F \maps \R \to \R'$ is a map of 2-rigs and $\R$ is semisimple, then $F$ factors up to 2-isomorphism in $\TRig$ as 
\[
\R \to \P \to \R'
\]
where $\P$ is a semisimple 2-rig, $\R \to \P$ is essentially surjective and full, and $\P \to \R'$ is faithful. The functors $\R \to \P$ and $\P \to \R'$ are exact. 
\end{thm}

To understand this statement, we need to understand the 2-category of 2-rigs and also the concept of a semisimple 2-rig.

\begin{defn}
\label{defn:2-rig}
    Let $\TRig$ denote the 2-category whose
    \begin{itemize}
        \item objects are \define{2-rigs}: symmetric monoidal Cauchy complete $k$-linear categories,  
        \item morphisms are \define{maps of 2-rigs}: symmetric monoidal $k$-linear functors, 
        \item 2-morphisms are symmetric monoidal $k$-linear natural transformations.
    \end{itemize}
\end{defn}

To define semisimplicity for 2-rigs, first recall that an object of an abelian category is defined to be simple if it has no nontrivial quotients, and semisimple if it is a coproduct of simple objects.  A $k$-algebra $R$ is said to be semisimple if it is semisimple as a left $R$-module (or, it turns out, equivalently, a right $R$-module).  Then:

\begin{defn}
A $k$-linear category $\R$ is \define{semisimple} if it is Cauchy complete and every endomorphism algebra $\R(A, A)$ is semisimple.
\end{defn}

In what follows we use some standard facts \cite{AndersonFuller,EGNO}: 

\begin{enumerate}
    \item A quotient of a semisimple algebra is semisimple.
    \item By Wedderburn--Artin theory, every semisimple algebra is a finite product of algebras, each of which is isomorphic to a matrix ring $M_n(D)$ where $D$ is a division algebra over $k$. 
    \item A semisimple $k$-linear category is the same as an abelian category in which all exact sequences split and every object is a finite coproduct of simple objects. 
\end{enumerate}

The following result is immediate from the splitting of exact sequences in semisimple categories: 

\begin{lem}
    If $F \maps \R \to \R'$ is a $k$-linear functor and $\R$ is semisimple, then $F$ is exact, i.e., it preserves exact sequences. 
\end{lem}

For any functor between categories $F \maps \R \to \R'$, not necessarily $k$-linear, there is factorization of $F$ as 
\[
\R \to \P \to \R'
\]
where the first functor $\R \to \P$ is essentially surjective and full (denote this by $\R \twoheadrightarrow \P$), and the second functor $\P \to \R'$ is faithful (denote this by $\P \rightarrowtail \R'$). This factorization is uniquely determined up to categorical equivalence. An explicit description is as follows: 

\begin{itemize}
    \item Objects of $\P$ are the objects of $\R$, and the functor $\R \to \P$ is the identity on objects; 
    \item Morphisms of $\P$ are equivalence classes of morphisms of $\R$, where two morphisms $f, g$ of the form $A \to B$ are equivalent, $f \sim g$, if $F(f) = F(g)$. The functor $\R \to \P$ takes a morphism $f$ of $\R$ to its equivalence class $[f]$. 
\end{itemize} 

Thus there are local epi-mono factorizations of the canonical maps $\R(A, B) \to \R'(FA, FB)$ between homsets, 
\[
\R(A, B) \twoheadrightarrow \P(A, B) \rightarrowtail \R'(FA, FB).
\]

Now assume $F$ is a map of 2-rigs, and that $\R$ is semisimple. Here are some basic facts about the (es+full, faithful) factorization $\R \twoheadrightarrow \P \rightarrowtail \R'$ of $F$. 

\begin{lem}
    $\P$ is a $k$-linear category, and the functors $\R \to \P$ and $\P \to \R'$ are $k$-linear. 
\end{lem}

\begin{proof}
    $\P$ acquires the structure of $k$-linear category, and the functors $\R \to \P$ and $\P \to \R'$ are $k$-linear functors, since the epi-mono factorizations of the $k$-linear maps $\R(A, B) \to \R'(FA, FB)$ lift from sets to vector spaces. 
\end{proof}

\begin{lem}
    The functor $\R \to \P$ preserves monos and epis. 
\end{lem}

\begin{proof}
    The exact functor $F \maps \R \to \R'$ preserves monos and epis: if a morphism $f$ in $\R$ is monic/epic, then $F(f)$ is monic/epic in $\R'$. The image $[f]$ in $\P$ is monic/epic in $\P$, because $F(f)$ is monic/epic in $\R'$, and any faithful functor, for instance $\P \to \R'$, reflects monos and epis. 
\end{proof} 

\begin{lem}
    $\P$ is Cauchy complete as a $k$-linear category. 
\end{lem}

\begin{proof}
    If $A \oplus B$ is a biproduct of $A$ and $B$ in $\R$, with injections and projections 
    \[
    i_A \maps A \to A \oplus B, \qquad i_B \maps B \to A \oplus B, \qquad p_A \maps A \oplus B \to A, \qquad p_B \maps A \oplus B \to B,
    \]
    satisfying the equations $p_A i_A = 1_A$, $p_B i_B = 1_B$, $p_A i_B = 0$, $p_B i_A = 0$, $i_A p_A + i_B p_B = 1_{A \oplus B}$, then the same equations transfer along the $k$-linear functor $\R \to \P$, making $A \oplus B$ the biproduct of $A$ and $B$ in $\P$. 

    Let $[e] \maps A \to A$ be an idempotent in $\P$. This is the image of some map $e \maps A \to A$ in $\R$, not necessarily idempotent. But $e$ does have a mono-epi factorization in $\R$, 
    \[
    A \overset{r}{\twoheadrightarrow} \im(e) \overset{i}{\rightarrowtail} A.
    \]
    Note that $[r]$ is epi in $\P$ and $[i]$ is monic in $\P$, since the functor $\R \to \P$ preserves monos and epis. The pair $[r], [i]$ splits $[e]$, since $[e] = [i] [r]$ by functoriality, and also 
    \[
    [i][r][i][r] = [e][e] = [e] = [i][1_{\im(e)}][r]
    \]
    whence $[r][i] = [1_{\im(e)}]$ follows from $[r]$ being epic and $[i]$ being monic. 
\end{proof}

\begin{lem}
    The $k$-linear category $\P$ is semisimple. 
\end{lem}

\begin{proof}
    $\P$ is Cauchy complete, so by our definition of semisimple $k$-linear categories it only remains to show that each endomorphism algebra $\P(A, A)$ is semisimple. But $\P(A, A)$ is a quotient of the semisimple algebra $\R(A, A)$, and as earlier observed, quotients of semisimple algebras are semisimple. 
\end{proof}

\begin{lem}
    $\P$ is a 2-rig, and the functors $\R \to \P$ and $\P \to \R'$ are 2-rig maps. 
\end{lem} 

\begin{proof}
    The monoidal product $\otimes \maps \P \times \P \to P$ is defined objectwise as in $\R$. For morphisms $[f] \maps A \to C$ and $[g] \maps B \to D$ in $\P$, define $[f] \otimes [g] \maps A \otimes B \to C \otimes D$ to be $[f \otimes g]$. To see this is well-defined (is independent of the $f, g$ in $\R$ chosen to represent $[f], [g]$), let 
    \[
    \phi_{AB} \maps FA \otimes FB \to F(A \otimes B)
    \]
    be the (invertible) structural constraint on the symmetric monoidal functor $F$. Then $F(f \otimes g)$ is uniquely determined from $Ff$ and $Fg$ via the composite 
    \[
    \begin{tikzcd}
        F(A \otimes B) \ar[r, "\phi_{AB}^{-1}"] & FA \otimes FB \ar[r, "Ff \otimes Fg"] & FC \otimes FD \ar[r, "\phi_{CD}"] & F(C \otimes D).
    \end{tikzcd}
    \]
    Well-definedness of $[f \otimes g]$ then follows from faithfulness of the functor $J \maps \P \to \R'$ in the factorization, where $Ff = J([f])$ and $Fg = J([g])$ and $F(f \otimes g) = J([f \otimes g]$ for uniquely determined maps $[f], [g], [f \otimes g]$ in $\P$, hence $[f \otimes g]$ is uniquely determined from $[f]$ and $[g]$. All structural constraints (associativity, symmetry, etc.) for the tensor product on $\P$ descend from those in $\R$, and the rest of the proof is routine, again taking advantage of faithfulness of $J$, where all necessary equations that must hold in $\P$ (functoriality and $k$-linearity of $\otimes$, etc.) are reflected from the corresponding equations holding in $\R'$.     
\end{proof}

The last two propositions prove the theorem at the head of this section, \cref{thm:quotient_2-rig}, but a few more concepts will be useful for us. 


For $F \maps \R \to \R'$ a 2-rig map, define the \define{kernel} of $F$ to be the full subcategory of $\R$ consisting of objects $r \in \R$ such that $F(r) \cong 0$. It is easy to see that if $F$ is faithful, then $\ker(F)$ consists only of zero objects. Clearly $\ker(F)$ also has the following properties: 

\begin{itemize}
    \item It is closed under biproducts and retracts in $\R$. 
    \item It is \define{replete}: if $r \in \ker(F)$ and $r \cong s$, then $s \in \ker(F)$. 
    \item If $r \in \ker(F)$ and $s \in \R$, then $r \otimes s \in \ker(F)$. 
\end{itemize}

We define a \define{2-ideal} of $\R$ to be a full subcategory of $\R$ having these properties. In the special case where $\R$ is semisimple, we may adduce a few more properties. 

\begin{lem}
    If $\R$ is semisimple, then any 2-ideal $\I$ in $\R$ is a Serre subcategory (it is closed under subobjects, quotients and extensions). Moreover, $\I$ is the (replete) finite coproduct closure of the simple objects it contains. 
\end{lem}

\begin{defn}
\label{defn:quotient_2-rig}
Let $\I$ be a 2-ideal of a semisimple 2-rig $\R$. Define the \define{quotient} $\R/\I$ to be the following category:

\begin{itemize}
    \item The objects of $\R/\I$ are the objects of $\R$, 
    \item The morphisms of $\R/\I$ are equivalence classes of morphisms of $\R$, where $f, g \maps a \to b$ are equivalent, $f \sim g$, if $f - g$ factors through an object $I \in \I$. (Equivalently by the preceding lemma, if $\im(f-g) \in \I$, since the image will be a subquotient of any $I$ that $f-g$ factors through.) 
\end{itemize}
\end{defn}

A few things have to be checked, of course. Obviously $\sim$ is reflexive and symmetric. It is transitive because if $f-g$ factors through $I \in \I$ and $g - h$ factors through $I' \in \I$, then their sum $f - h$ factors through $I \oplus I' \in \I$. If $f \sim g \maps A \to B$ and $h \maps B \to B'$, then $h f \sim h g$ since $hf - hg = h(f-g)$ factors through any object that $f-g$ factors through; dually, $f h \sim g h$ for any $h \maps A' \to A$. It follows that $\sim$ is a categorical congruence: the category structure on $\R$ descends to a category structure on $\R/\I$. 

Through similarly routine arguments, largely parallel to arguments given earlier in this section and left for the reader to verify, the following result holds: 

\begin{lem}
    If $\I$ a 2-ideal of a semisimple 2-rig $\R$, $\R/\I$ inherits from $\R$ a 2-rig structure, so that the quotient functor $\R \to \R/\I$ is a 2-rig map. $\R/\I$ is a semisimple 2-rig. 
\end{lem}

Another way to view $\R/\I$ is that it is the localization of $\R$ obtained by formally inverting zero maps $I \to 0$ for all $I \in \I$. This point of view is implicit in the proof of the next result, which could be called a ``first isomorphism theorem'' for 2-rig maps $F \maps \R \to \R'$ (again assuming $\R$ is semisimple). 

\begin{lem}
    Given a semisimple 2-rig $\R$ and a 2-rig map $F \maps \R \to \R'$, with (es+full, faithful) factorization $\R \to \P \to \R'$, the 2-rig $\P$ is identified with $\R/\ker(F)$. 
\end{lem} 

\begin{proof}
Both $\P$ and $\R/\ker(F)$ are declared to have the same objects as $\R$, and in both cases the morphisms are defined as equivalence classes of morphisms in $\R$; it is simply a matter of checking that the equivalence relations are the same. We have that $[f] = [g]$ in $\P$ iff $F(f) = f(g)$ iff $F(f-g) = 0$. The claim is that the last is equivalent to $F(\im(f-g)) \cong 0$, which by definition means $\im(f-g) \in \ker(F)$; as observed earlier, this last condition is equivalent to the condition that $f$ and $g$ are equivalent in $\R/\ker(F)$, which would complete the proof. Thus all that remains is to verify the claim. For brevity, put $h = f-g$. Applying $F$ to the epi-mono factorization of $h$, and the fact that $F$ preserves monos and epis, leads to 
\[
    F(h) = (FA \overset{p}{\twoheadrightarrow} F(\im(h)) \overset{i}{\rightarrowtail} FB)
\]
and the hypothesis that the composite $F(h)$ is $0$ means $0 = i\circ p = i \circ 0$, whence $p = 0$ by monicity of $i$. If $p = 0 \maps FA \to F(\im(h))$ is epic, conclude that $1_{F(\im(h))} \circ p = 0 = 0_{F(\im(h))} \circ p$, whence $1_{F(\im(h))} = 0_{F(\im(h))}$, i.e., $F(\im(h)) \cong 0$, thus proving the claim.
\end{proof}

Notice that this result stands in stark contrast to the situation for ordinary (commutative) rigs: there is no first isomorphism theorem there, because the kernel of a rig map $f \maps R \to S$ generally will not suffice to describe the rig congruence on $R$ arising from the epi-mono factorization of $f$ through a quotient of $R$. This is the case even when addition in $R$ is cancellative. For example, for $R = \mathbb{N}$, the smallest rig congruence $\sim$ on $\mathbb{N}$ that identifies $2$ with $3$ induces a quotient rig map $q$ from $\mathbb{N}$ to a 3-element rig, but the kernel of that $q$ is zero. It is in recognition of this fact that we take some extra care over mundane details in the proof above, and not assume analogies based on a hasty interpretation of `categorification'. 

\section{Quotients of the free 2-rig on one generator}
\label{sec:quotients_of_the_free_2-rig}

By \cref{thm:quotient_2-rig}, if $\R$ is any semisimple 2-rig and $F \maps \R \to \Fin\Vect$ is a 2-rig map, then the (es+full, faithful) factorization of $F$ induces a 2-rig map 
\[
J \maps \R/\ker(F) \to \Fin\Vect
\]
and moreover $\R/\ker(F)$ is semisimple, while $J$ is faithful and exact. Therefore, by \cref{thm:Tannaka_2'}, $\R/\ker(F)$ is the category of finite-dimensional comodules of some commutative bialgebra. Equivalently, by \cref{lem:Comod}, it is the category $\Rep(M)$ of algebraic representations of some affine monoid $M$. In this section we explore the possibilities when $\R = \ksbar$, the free 2-rig on one generator. 

Because there is an equivalence 
\[
\TRig(\ksbar, \Fin\Vect) \simeq \Fin\Vect
\]
given by evaluating a 2-rig map $\phi \maps \ksbar \to \Fin\Vect$ at the generator $x$, such 2-rig maps are determined up to isomorphism by the isomorphism class of $V = \phi(x)$, which is in turn determined up to isomorphism by $n = \dim(V)$. Therefore 2-rig maps $\phi \maps \ksbar \to \Fin\Vect$ are classified by natural numbers $n$; we define $\phi_n \maps \ksbar \to \Fin\Vect$ to be the unique (up to isomorphism) 2-rig map that sends $x$ to $k^n$. We proceed to compute its (es+full, faithful) factorization, 
\[
\ksbar \twoheadrightarrow \ksbar/\ker(\phi_n) \rightarrowtail \Fin\Vect.
\]

Let $\Lambda^n(x) \in \ksbar$ be the $n$th exterior power of the generating object $x \in \ksbar$; this object corresponds to the sign representation of $S_n$. It generates a 2-ideal we denote as $\langle \Lambda^n \rangle$. In other words, $\langle \Lambda^n \rangle$ is defined to be the smallest 2-ideal containing $\Lambda^{n}(x)$. 

\begin{lem}
\label{lem:ker(phi_n)}
There is an equality of 2-ideals $\ker(\phi_n) = \langle \Lambda^{n+1} \rangle$. 
\end{lem} 

\begin{proof}
The inclusion $\langle \Lambda^{n+1} \rangle \subseteq \ker(\phi_n)$ follows from the fact that $\phi_n\langle \Lambda^{n+1} \rangle \cong 0$, which is clear because
\[
    \phi_n(\Lambda^{n+1}) \cong \Lambda^{n+1}(k^n) \cong 0
\]
in $\Fin\Vect$. 
    
To prove the inclusion $\ker(\phi_n) \subseteq \langle \Lambda^{n+1} \rangle$, start from the fact that a 2-ideal is uniquely determined by the simple objects it contains, indeed it is the finite coproduct closure of the class of its simple objects. The simple objects of $\langle \Lambda^{n+1} \rangle$ are those that occur as retracts of Schur objects $\Lambda^{n+1} \otimes x^{\otimes m}$. By a simple application of Pieri's rule, these simple objects correspond to partitions $\lambda_1 \geq \lambda_2 \geq \cdots \geq \lambda_{n+1} \geq \ldots$ (i.e., to Young diagrams with more than $n$ rows), so it remains to show that there is no simple object $\rho_\lambda$ in $\ker(\phi_n)$ that corresponds to a Young diagram $\lambda$ with at most $n$ rows, or to a partition $\lambda_1 \geq \cdots \geq \lambda_n$. That is to say, that $\phi_n(\rho_\lambda) \ncong 0$: this means that, letting $m$ be the number of boxes in $\lambda$, the vector space 
\[
    \phi_n(\rho_\lambda) = \rho_\lambda \otimes_{kS_m} (k^n)^{\otimes m}
\]
has positive dimension. But the dimension of this space equals the number of semistandard Young tableaux of type $\lambda$ with boxes labeled in $\{1, 2, \ldots, n\}$, and obviously there is at least one such tableau (for example, the one where all boxes in row $k$, for $1 \leq k \leq n$, are labeled by $k$). This completes the proof.
\end{proof}

\section{Proof of the main theorem}
\label{sec:main_theorem}

We have seen that there is a map of 2-rigs
\[   \phi_n \maps \ksbar \to \Fin\Vect,  \]
unique up to isomorphism, sending the generating object $x$ of the free 2-rig on one object to $k^n \in \Fin\Vect$, and that the kernel of $\phi_n$ is the 2-ideal generated by $\Lambda^{n+1}$.   It follows that the (es+full, faithful) factorization of $\phi_n$ is the evident pair
\[
\ksbar \overset{q}{\twoheadrightarrow} \ksbar/\langle \Lambda^{n+1} \rangle \overset{j}{\rightarrowtail} \Fin\Vect.
\]
We call the object $q(x)$ simply $x$, since $\ksbar/\langle \Lambda^{n+1} \rangle$ has the same objects as $\ksbar$ and $q$ is the identity on objects.

\begin{lem}
\label{lem:free_2-rig}
$\ksbar/\langle \Lambda^{n+1} \rangle$ is the free 2-rig on an object of subdimension $n$. That is to say, it is the representing object for the functor $\mathrm{Subdim}_n \maps \TRig \to \Cat$ that assigns to each 2-rig $\R$ the full subcategory of objects of subdimension $n$ in $\R$. In particular, given any 2-rig $\R$ containing an object $r$ of subdimension $n$, there is a map of 2-rigs $F \maps \ksbar/\langle \Lambda^{n+1} \rangle \to \R$ with $F(x) = r$, unique up to uniquely determined monoidal natural isomorphism.
\end{lem}

\begin{proof}
Evaluation at the object $x \in \ksbar$ induces a 2-rig map $\TRig(\ksbar/\langle \Lambda^{n+1}\rangle, \R) \to \R$ given by the composite 
\[\TRig(\ksbar/\langle \Lambda^{n+1}\rangle, \R) \to \TRig(\ksbar, \R) \overset{\mathrm{ev}_x}{\longrightarrow} \R.\]
All values of this functor are objects of subdimension $n$, since for any 2-rig map $A \maps \ksbar/\langle \Lambda^{n+1}\rangle \to \R$, the object $A(x)$ satisfies $\Lambda^{n+1}(A(x)) \cong A(\Lambda^{n+1}) \cong A(0) \cong 0$ because $A$ preserves exterior powers and $\Lambda^{n+1} \cong 0$ as an object in $\ksbar/\langle \Lambda^{n+1} \rangle$.  We thus have a 2-rig map
\[  \begin{array}{rcl}
\mathrm{ev} \maps \TRig(\ksbar/\langle \Lambda^{n+1}\rangle, \R) &\to& \mathrm{Subdim}_n(\R) \\
A & \mapsto & A(x).
\end{array}
\]

To prove that $\ksbar/\langle \Lambda^{n+1} \rangle$ is the representing object for $\mathrm{Subdim}_n$ we shall show that $\mathrm{ev}$ is an equivalence.  We begin by showing it is essentially surjective. If $r \in \mathrm{Subdim}_n(\R)$, there is a 2-rig map $G \maps \ksbar \to \R$ with $G(x) = r$, and from $\Lambda^{n+1}(r) \cong 0$ we deduce $G(\Lambda^{n+1}) \cong \Lambda^{n+1}(G(x)) = \Lambda^{n+1}(r) \cong 0$.  It follows that $\langle \Lambda^{n+1} \rangle \subseteq \ker G$, since $\langle \Lambda^{n+1} \rangle$ is by definition the smallest 2-ideal containing $\Lambda^{n+1}$. This implies the well-definedness of the functor $F \maps \ksbar/\langle \Lambda^{n+1}\rangle \to \R$ for which $F(\rho) = G(\rho)$ for all objects $\rho$ of $\ksbar$ and $F([f]) = G(f)$ for all morphisms $f$ of $\ksbar$. Indeed, if $[f] = 0$, then $\im(f) \in \langle \Lambda^{n+1} \rangle$ by definition of the quotient 2-rig, so $\im(f) \in \ker(G)$, i.e., $G(\im(f)) \cong 0$, whence $G(f) = 0$. The functor $F$ can be shown to be a 2-rig map using the fact that $G$ is a 2-rig map. Thus $\textrm{ev}$ is essentially surjective. 

This map $\textrm{ev}$ is also full and faithful. Indeed, the functor from which it descends, 
\[
\textrm{ev}_x \maps \TRig(\ksbar, \R) \to \R,
\]
again given by evaluation at $x$, is full and faithful. So, for each $f \maps r \to r'$ in $\mathrm{Subdim}_n(\R)$, with $r = G(x)$ and $r' = G'(x)$ for $G, G' \in \TRig(\ksbar, \R)$, there is a unique monoidal natural transformation $\phi \maps G \To G'$ such that $\phi(x) = f$. The map $\phi$ descends to a monoidal natural transformation $\psi \maps F \To F'$ in $\TRig(\ksbar/\langle \Lambda^{n+1} \rangle, \R)$, evidently unique, for which $\psi(x) = f$. This establishes full faithfulness. 

The final statement of the lemma is a simple consequence of the representability statement. If $F, G \maps \ksbar/\langle \Lambda^{n+1} \rangle \rightrightarrows \R$ are two 2-rig maps such that $F(x) = G(x) = r$, then evaluation at $x$, being full and faithful, induces an isomorphism of hom-sets 
\[
\TRig(\ksbar/\langle \Lambda^{n+1} \rangle, \R)(F, G) \overset{\textrm{ev}_x}{\longrightarrow} \mathrm{Subdim}_n(\R)(r, r),
\]
and the unique 2-cell $\alpha \maps F \to G$ that maps to $1_r$ is the asserted monoidal natural transformation. 
\end{proof}

We are now ready to prove the main theorem:

\begin{theorem1}
The 2-rig $\Rep(\M(n,k))$ is the free 2-rig on an object of subdimension $n$.  That is, given any 2-rig $\R$ containing an object $r$ of subdimension $n$, there is a map of 2-rigs $F \maps \Rep(\M(n,k)) \to \R$ with $F(k^n) = r$, unique up to a uniquely determined monoidal natural isomorphism.
\end{theorem1}

\begin{proof}
By \cref{lem:free_2-rig} we know that $\ksbar/\langle \Lambda^{n+1} \rangle$ is the free 2-rig on an object of subdimension $n$.  It thus suffices to show that
$\ksbar/\langle \Lambda^{n+1} \rangle$ is equivalent, as a 2-rig, to $\Rep(\M(n,k))$.  In fact we shall show that
\[  j \maps \ksbar/\langle \Lambda^{n+1} \rangle \to \Fin\Vect  \]
is isomorphic in $\TRig \downarrow \Fin\Vect$ to
\[   W \maps \Rep(\M(n,k)) \to \Fin\Vect \]
where $W$ is the forgetful functor.   Thanks to \cref{lem:commuting_triangle_of_2-rigs} we know that the latter is 
isomorphic to
\[ U \maps \Comod(\O(\M(n,k))) \to \Fin\Vect \]
which by \cref{ex:full_linear_monoid} is equal to
\[ U \maps \Comod(\Sym(V^\ast \otimes V)) \to \Fin\Vect \]
where $V = k^n$.   Thus, it suffices to show 
\[  j \maps \ksbar/\langle \Lambda^{n+1} \rangle \to \Fin\Vect  \]
is isomorphic to
\[  U \maps \Comod(\Sym(V^\ast \otimes V)) \to \Fin\Vect \]
in $\TRig \downarrow \Fin\Vect$.

Since $j$ is a faithful, exact 2-rig map we know by Tannaka reconstruction that it is equivalent to the forgetful functor
\[    \Comod(B) \to \Fin\Vect \]
for \emph{some} commutative bialgebra $B$.   In fact we know by \cref{thm:Tannaka_2'} that $B$ is the coend $\End(j)^{\vee}$.   Thus, to prove the main theorem, we just need to show that
\[         \End(j)^\vee \cong \Sym(V^\ast \otimes V). \]
We turn to this now.

Note that $\phi_n$ factors as 
\[
\ksbar \to \Rep(\M(n,k)) \to \Fin\Vect
\]
where $\Rep(\M(n,k)) \simeq \Comod(\Sym(V^\ast \otimes V))$. Since the second 2-rig map above is faithful, there is a further factoring as 
\[
\ksbar \to \ksbar/\langle \Lambda^{n+1} \rangle \to \Rep(\M(N,k)) \overset{i}{\to} \Fin\Vect
\]
where the latter two arrows are faithful 2-rig maps, equivalently written in the form 
\[
\ksbar/\langle \Lambda^{n+1} \rangle \xrightarrow{\alpha} \Comod(\Sym(V^\ast \otimes V)) \xrightarrow{i} \Fin\Vect.
\]
We can regard $\alpha$ as a morphism from $j$ to $i$ in $\LinCat \downarrow \Fin\Vect$.  We now show that the canonical bialgebra map $\End(j)^{\vee} \to \Sym(V^\ast \otimes V)$ adjoint to this morphism $\alpha \maps j \to i$ is an isomorphism.   This will complete the proof.

The calculation of the coend 
\[   \End(j)^{\vee} = \int^{R: \ksbar/\langle \Lambda^{n+1} \rangle}\; j(R)^\ast \otimes j(R) \]
is much simplified by the following observations: 
\begin{itemize}
\item Since the $j(R)$ are finite-dimensional, the coend may be written in the form 
\[
\int^{R: \ksbar/\langle \Lambda^{n+1} \rangle}\; \Fin\Vect(j(R), j(R))
\]
\item In this formula, $\ksbar/\langle \Lambda^{n+1} \rangle$ can be replaced by $\ksbar$, and $j$ by $\phi_n$. This is a simple consequence of essential surjectivity and fullness of the quotient map $q \maps \ksbar \to \ksbar/\langle \Lambda^{n+1} \rangle$. Indeed, very generally, any coend $\int^{c \in \C} F(c, c)$ is a coequalizer of an evident pair of maps 
\[
\sum_{c, c'} \C(c, c') \otimes F(c', c) \rightrightarrows \sum_c F(c, c).
\]
Assuming WLOG that $q \maps \D \to \C$ is full and an identity on objects, then $\int^{d \in \D} F(qd, qd) \cong \int^{c \in \C} F(c, c)$, because the coend over $d$ is the coequalizer of parallel composites 
\[
\sum_{d = c, d' = c'} \D(d, d') \otimes F(qd', qd) \overset{\pi}{\longrightarrow} \sum_{c, c'} \C(c, c') \otimes F(c', c) \rightrightarrows \sum_c F(c, c)
\]
where $\pi$ is an epimorphism induced by surjective maps $\D(d, d') \twoheadrightarrow \C(qd, qd')$ (fullness of $q$). But the coequalizer of a pair of maps $a\pi, b\pi$ is isomorphic to the coequalizer of $a, b$ on condition that $\pi$ is epic. In our situation, this gives 
\[
\begin{array}{ccl}
\int^{R: \ksbar/\langle \Lambda^{n+1} \rangle}\; \Fin\Vect(j(R), j(R)) &\cong&
\int^{R \in \ksbar} \Fin\Vect(jq(R), jq(R)) \\ \\
&=& \int^{R \in \ksbar} \Fin\Vect(\phi_n(R), \phi_n(R)).
\end{array}
\]
\item This type of coend, related to the {\em trace} of an enriched category, can be seen as a composite of $\Vect$-enriched profunctors 
\[
        \begin{tikzcd}
            k \ar[rr, "\hom_{\Fin\Vect}"] & & \Fin\Vect\op \otimes \Fin\Vect \ar[rr, "(\phi_n\op \otimes \phi_n)^\ast"] & & \ksbar\op \otimes \ksbar \ar[r, "\int^{\ksbar}"] & k.
        \end{tikzcd}
        \]
\item Since equivalences in the bicategory of profunctors are categorical Morita equivalences, it is harmless to replace this coend, as a profunctor composite, by the corresponding coend
        \[
        \int^{x^{\otimes m} \in k\mathsf{S}} \Fin\Vect(\phi_n(x^{\otimes m}), \phi_n(x^{\otimes m}))
        \]
where the coend is now over the full subcategory $k\S$ of $\ksbar$ consisting of only the representable objects $x^{\otimes n}$ of $\ksbar$, which is Morita equivalent to $\ksbar$. 
\item As $\phi_n$ preserves tensor products, and $\phi_n(x) = V = k^n$, the hom-spaces in the previous coend may be rewritten as $\Fin\Vect((k^n)^{\otimes m}, (k^n)^{\otimes m})$. 
\end{itemize}

The object $x^{\otimes m}$ appearing in the coend superscript is the regular representation of the group algebra $kS_m$. Since the only arrow of type $kS_m \to kS_n$ in $k\S$ (for $m$ different from $n$) is the zero arrow, it follows that the last coend breaks up as a coproduct 
    \[
    \sum_{m \geq 0} \int^{x^{\otimes m} \in kS_m} \Fin\Vect((k^n)^{\otimes m}, (k^n)^{\otimes m}).
    \]
This may be rewritten as 
    \[
    \sum_{m \geq 0} \int^{kS_m} ([k^n]^{\otimes m})^\ast \otimes (k^n)^{\otimes m} = \sum_{m \geq 0} ([k^n]^{\otimes m})^\ast \otimes_{kS_m} (k^n)^{\otimes m}.
    \]
Finally, this last expression is isomorphic to $\sum_{m \geq 0} \mathrm{Sym}^m ((k^n)^\ast \otimes (k^n))$, which is the symmetric algebra $\Sym(V^\ast \otimes V)$.   This completes the proof that
\[         \End(j)^\vee \cong \Sym(V^\ast \otimes V).  \qedhere \]
\end{proof}

\section{Conclusions}
\label{sec:conclusions}

The work here suggests that a number of related questions are now within reach.  One was left as a conjecture in our previous paper \cite[Conj.\ 8.8]{Splitting}:

\begin{conj}
\label{conj:free_2-rig_on_object_of_dimension_N} 
If $k$ is a field of characteristic zero, the 2-rig $\Rep(\GL(n,k))$ is the free 2-rig on an object of dimension $n$, that is, an object $x$ for which $\Lambda^n(x)$ has an inverse with respect to the tensor product.
\end{conj}

This requires further techniques beyond what we have developed here, since \break $\Rep(\GL(n,k))$ is not a mere quotient of $\ksbar$, but it would be interesting to develop these techniques, which may let us characterize the representation 2-rigs of other so-called `classical' groups \cite{Weyl}.  We expect the following conjectures to hold for any field $k$ of characteristic zero:

\begin{conj}
\label{conj:Rep(SL(n,k))} 
If $k$ is a field of characteristic zero, the 2-rig $\Rep(\SL(n,k))$ is the free 2-rig on an object $x$ equipped with an isomorphism $\Lambda^n(x) \cong I$, where $I$ is the unit for the tensor product.
\end{conj}

\begin{conj}
\label{conj:Rep(Sp(n,k))} 
If $k$ is a field of characteristic zero, the 2-rig $\Rep(\Sp(n,k))$ is the free 2-rig on a self-dual object $x$ of dimension $n$ whose counit $\epsilon \maps x \otimes x \to I$ is antisymmetric: $\epsilon \circ S_{x,x} = -\epsilon$.
\end{conj}

In the remaining conjectures we assume $k$ is algebraically closed, so that all nondegenerate symmetric bilinear forms on an $n$-dimensional vector space over $k$ are isomomorphic.

\begin{conj}
\label{conj:Rep(O(n,k))} 
If $k$ is an algebraically closed field of characteristic zero, the 2-rig $\Rep(\OO(n,k))$ is the free 2-rig on a
self-dual object $x$ of dimension $n$ whose counit $\epsilon \maps x \otimes x \to I$ is symmetric: $\epsilon \circ S_{x,x} = \epsilon$.
\end{conj}

\begin{conj}
\label{conj:Rep(SO(n,k))} 
If $k$ is an algebraically closed field of characteristic zero, the 2-rig $\Rep(\SO(n,k))$ is the free 2-rig on an object $x$ that is equipped with an isomorphism $\Lambda^n(x) \cong I$ and is also self-dual with symmetric counit $\epsilon \maps x \otimes x \to I$.
\end{conj}

Going beyond the classical groups, it seems Cvitanovic \cite{Cvitanovic} has shown that the 2-rigs of representations of exceptional groups can be characterized by subtler universal properties involving cubic and quartic forms $x^{\otimes 3} \to I$ and $x^{\otimes 4} \to I$.

In another direction, it would be interesting to see how the story presented here changes over fields of nonzero characteristic.  In this paper, and the papers this one relies on \cite{Schur,Splitting}, we made heavy use of the fact that working in characteristic zero, categories of finite-dimensional representations of symmetric groups are semisimple.  This fails in nonzero characteristic.

\end{document}